\documentclass[]{asl}

\usepackage{enumerate,url}

\title[Determinacy Separations for Class Games]{Determinacy Separations for Class Games}

\author{Sherwood Hachtman}
\revauthor{Hachtman, Sherwood}

\address{Department of Mathematics, Statistics, and Computer Science\\
University of Illinois at Chicago\\
Chicago, IL 60613, USA}

\email{hachtma1@uic.edu}

\thanks{I am grateful to Victoria Gitman for introducing me to the questions discussed here.  I also thank the American Institute of Mathematics and organizers of the workshop ``High and Low Forcing'' held in January, 2016, which allowed these initial conversations to take place.}

\newtheorem{Theorem}{Theorem}[section]

\newtheorem{Lemma}[Theorem]{Lemma}

\newtheorem{Claim}[Theorem]{Claim}

\theoremstyle{definition}
\newtheorem{Definition}[Theorem]{Definition}

\newtheorem{Conjecture}[Theorem]{Conjecture}

\newtheorem{Remark}[Theorem]{Remark}

\renewcommand{\P}{\mathcal{P}}

\newcommand{\Coll}{\operatorname{Coll}}
\newcommand{\om}{\omega}
\newcommand{\la}{\langle}
\newcommand{\ra}{\rangle}
\newcommand{\C}{\mathcal{C}}

\newcommand{\clopdetR}{\Delta^\mathbb{R}_1\operatorname{-DET}}
\newcommand{\opdetR}{\Sigma^\mathbb{R}_1\operatorname{-DET}}
\newcommand{\clopdetX}{\Delta^X_1\operatorname{-DET}}
\newcommand{\opdetX}{\Sigma^X_1\operatorname{-DET}}
\newcommand{\clopdetk}{\Delta^\kappa_1\operatorname{-DET}}
\newcommand{\opdetk}{\Sigma^\kappa_1\operatorname{-DET}}

\newcommand{\ZF}{\mathsf{ZF}}

\newcommand{\ATR}{\mathsf{ATR}}
\newcommand{\ZFC}{\mathsf{ZFC}}

\newcommand{\KP}{\mathsf{KP}}

\newcommand{\wfo}{\operatorname{wfo}}

\newcommand{\ON}{\operatorname{ON}}

\newcommand{\plto}{\rightharpoonup}
\newcommand{\rst}{{\upharpoonright}}

\newcommand{\NBG}{\mathsf{NBG}}
\newcommand{\MK}{\mathsf{MK}}

\newcommand{\dom}{\operatorname{dom}}

\newcommand{\PCA}{\Pi^1_1\text{-}\mathsf{CA}_0}
\newcommand{\RCA}{\mathsf{RCA}^3_0}

\newcommand{\crit}{\operatorname{crit}}

\newcommand{\R}{\mathbb{R}}

\renewcommand{\epsilon}{\varepsilon}
\begin{document}

\begin{abstract}We show, assuming weak large cardinals, that in the context of games played in a proper class of moves, clopen determinacy is strictly weaker than open determinacy.  The proof amounts to an analysis of a certain level of $L$ that exists under large cardinal assumptions weaker than an inaccessible.  Our argument is sufficiently general to give a family of determinacy separation results applying in any setting where the universal class is sufficiently closed; e.g., in third, seventh, or $(\omega+2)$th order arithmetic.   We also prove bounds on the strength of Borel determinacy for proper class games.  These results answer questions of Gitman and Hamkins.\end{abstract}

\maketitle

\section{Introduction}  
One theme in the study of infinite games is their close connection to principles of transfinite recursion.  At the lowest level, this is embodied in Steel's seminal result \cite{St} that open and clopen determinacy for games on $\omega$ are both equivalent to the axiom $\ATR_0$ of arithmetic transfinite recursion.  In contrast, Schweber \cite{Schweber} has shown that in the third order setting, clopen determinacy for games with moves in $\R$ is equivalent to transfinite recursion along wellfounded relations on $\R$ (modulo some choice); but these principles do not imply open determinacy for games with moves in $\R$.  Recently we presented in \cite{HaRev} an alternate proof of Schweber's separation result using inner models in place of forcing.

The main result of this paper is a similar determinacy separation in the context of proper class games within some second or higher order set theory, such as von Neumann-Bernays-G\"{o}del set theory, $\NBG$.  Games of this kind were defined and investigated by Gitman and Hamkins \cite{GitHam}.  There they proved clopen determinacy is equivalent to a transfinite recursion principle allowing the iteration of first-order definitions along proper class wellorders; they conclude with a number of open questions, including that of whether open and clopen determinacy are equivalent for games with a proper class of moves.


We show here that a translation of our analysis in \cite{HaRev} can be employed to prove that clopen determinacy for class games does \emph{not} imply open determinacy for class games.  Indeed, our presentation uniformly separates open and clopen determinacy, not just for games on reals or on class trees in $\NBG$, but in many settings of typed higher order arithmetic, provided the next-to-largest type is sufficiently closed to allow coding of functions.  This answers Gitman and Hamkin's question in the negative, as well as reproving our generalization of Schweber's separation result to $n$-th order arithmetic, for all $n \geq 3$ (and indeed, to $(\alpha+2)$th order arithmetic for ordinals $\alpha$).

In the paper's final section, we discuss the strength of determinacy for class games with payoff in levels of the Borel hierarchy in the setting of Morse-Kelley set theory, $\MK$.  In particular, we show (again under mild large cardinals) Borel determinacy for proper class games is not provable in $\MK$.  This is a class games analogue of H.\ Friedman's famous result \cite{Fr} from the study of second order arithmetic.

We largely wish to avoid the formalisms of the typed set theories our results concern, and do so by focusing on a purely set-theoretic analysis of G\"{o}del's $L$ carried out in $\ZFC$.  The main prerequisite, then, is familiarity with constructibility and Condensation arguments; though it will be mentioned, the reader need not be familiar with Jensen's fine structure theory to understand this paper.

\section{Background and definitions}\label{sec:BasicDef}

\subsection{Games on trees}

Let $X$ be a set. $X^{<\om}$ denotes the set of finite sequences of elements of $X$; $X^Y$ denotes the set of functions $f:Y \to X$; hence $X^\om$ is the set of $\omega$-sequences of elements of $X$.  By \emph{tree on $X$}, we mean a subset $T \subseteq X^{<\om}$ closed under taking initial segments.  A tree is \emph{illfounded} if there exists an infinite branch through $T$, that is, an $x \in {}^\om X$ such that $\la x(0),\dots, x(n) \ra \in T$ for all $n$; otherwise $T$ is \emph{wellfounded}.  

We regard games as being played on trees.  If $T$ is a tree on $X$, two players, I and II, play a game on $T$ by alternating choosing elements of $X$, e.g.\
  \[
  \begin{array}{l|llllll}
  \text{I}  & x_0 &     & x_2 & \dots & x_{2n} &         \\
  \hline
  \text{II} &     & x_1 &     & \dots &      & x_{2n+1}
  \end{array},
  \]
subject to the rule that $\la x_0,\dots, x_i \ra \in T$ at all positions of the game.  The first player who disobeys this requirement (by being forced to make a move from a terminal node of $T$, for example) loses the game.

We understand a \emph{strategy} to be a special kind of function, with domain a subset of $T$ and codomain $X$, which instructs one of the two players how to move at all positions reachable according to the strategy at which it is that player's turn.  Formally, however, we here regard strategies as subtrees of $T$ of a special form, e.g.\ a strategy for Player I is a tree $S \subseteq T$ so that whenever $s \in S$ has even length, there is exactly one $x \in X$ so that $s^\frown\la x \ra \in S$; and if $s$ has odd length, then $s^{\frown}\la x \ra \in S$ iff $s^{\frown} \la x \ra \in T$, for all $x \in X$.  So construed, strategies in games on $T$ may be regarded as subsets of $T$; in the event that $X$ is closed under taking finite sequences, strategies in games on $X$ are themselves subsets of $X$.

\emph{Clopen determinacy for games on $X$}, denoted $\clopdetX$, is the assertion that whenever $T$ is a wellfounded tree on $X$, one of the players has a winning strategy in the game on $T$ (i.e., a strategy that never instructs that player to leave $T$).  \emph{Open determinacy for games on $X$}, denoted $\opdetX$, is the assertion that for all trees $T$ on $X$, either Player I (Open) has a strategy in $T$ that contains no infinite plays (equivalently, is a wellfounded subtree of $T$), or Player II (Closed) has a strategy in $T$ that avoids terminal nodes.


\subsection{Models that rank their wellfounded trees}

All of the separation results of this paper are witnessed in levels of G\"{o}del's constructible universe $L$.  We prefer here to work with Jensen's $J$-hierarchy (see \cite{Jensen}, \cite{ZS}), though standard remarks about the harmlessness of ignoring the difference between $J_\alpha$ and $L_\alpha$ apply; in particular, $J_\alpha = L_\alpha$ when $\om\cdot \alpha = \alpha$.  The main fine structural feature we need is \emph{acceptability}, namely, whenever $\rho \leq \alpha$ and there is a subset of $\rho$ in $J_{\alpha}\setminus J_{\alpha+1}$, there is a surjection $h:\rho \to J_\alpha$ in $J_{\alpha+1}$.  Note: We adopt the indexing of the $J_\alpha$'s whereby $\ON \cap J_\alpha = \omega \cdot \alpha$ for ordinals $\alpha$.

It is a theorem of $\ZFC$ that a tree $T$ is wellfounded if and only if there is a function $\rho:T \to \ON$ such that whenever $s,t \in T$ with $s \subsetneq t$, we have $\rho(s) > \rho(t)$.  For any wellfounded tree, there is a unique such function which takes the minimal possible values, that is, $\rho(s) = \sup_{s \subsetneq t \in T} (\rho(t)+1)$; we say this $\rho$ is the \emph{ranking function of $T$}, and assert the existence of a ranking function by saying \emph{$T$ is ranked}.  We remark in passing that the existence of ranking functions for wellfounded trees is \emph{not} provable in weak theories such as $\KP$ (though cf.\ Theorem~\ref{thm:noComp} and Remark~\ref{rem:Adm}, below).

\begin{Definition}\label{def:theta} Let $\psi(v)$ be a $\Pi_1$ formula in the language of set theory with one free variable.  We let $\theta^{\psi}$ be the least ordinal $\theta$ (if there is one) so that $J_\theta$ satisfies: ``There is a largest cardinal, $\kappa$, and:
  \begin{itemize}
    \item $\psi(\kappa)$ holds;
    \item $\kappa$ is regular and uncountable;
    \item Every wellfounded tree on $\kappa$ is ranked.''
  \end{itemize}
\end{Definition}
In our main application, $\psi(\kappa)$ will express ``$\kappa$ is inaccessible''.  Wherever possible, we drop $\psi$ from the notation for simplicity, referring simply to $\theta$.  We describe how varying $\psi$ obtains various separation results in \S\ref{sec:Frags}.

Note that by Condensation, $\theta$, if it exists, is countable.  Note also that since $\kappa$ is the largest cardinal of $J_\theta$,  any tree in $J_\theta$ may be regarded as a tree on $\kappa$ by taking its image under an appropriate bijection.  Furthermore, countable closure of $\kappa$ in $J_\theta$ guarantees that if $T \in J_\theta$ is an illfounded tree on $\kappa$, then there is an infinite branch through $T$ which belongs to $J_\kappa$; thus the wellfoundedness of trees in $J_\theta$ is $\Sigma_0$ in parameters.

\begin{Lemma}\label{lem:thetaordclosure} $\theta$ is closed under ordinal successor, addition, multiplication, and exponentiation.  In particular, $\omega \cdot \theta = \theta$, and so $L_\theta = J_\theta$.
\end{Lemma}

\begin{proof} Recall (\cite{ZS}, Lemma 1.5) that each $J_\alpha$ is a model of $\Sigma_0$-Comprehension; also, clearly $\kappa^{<\omega} \in J_{\kappa+1} \subseteq J_\theta$.  Suppose $\kappa < \alpha < \omega \cdot \theta$.  By definition of $J_\theta$ there is a bijection $h_\alpha : \kappa \to \alpha$ which belongs to $J_\theta$.  Define
  \[
    T_\alpha := \{ s \in \kappa^{<\om} \mid s(0) = 0 \text{ and } h_\alpha \circ s(i) > h_\alpha \circ s(j)\text{ for all }0 < i < j < |s| \}.
  \]
The tree $T_\alpha$ is clearly wellfounded, and has rank at least $\alpha+1$; hence $\alpha+1 < \omega \cdot \theta$.  

For ordinal addition, if $\alpha, \beta < \omega \cdot \theta$ then defining $T_\alpha, T_\beta$ as above, 
  \[
    T_{\alpha + \beta} = \{ s^{\frown} t \mid s \in T_\beta, t \in T_\alpha\}
  \]
is wellfounded with rank $\alpha + \beta$.

The remaining closure properties are proved similarly, defining trees with suitably large ranks; we leave these as an exercise.
\end{proof}
Although $L_\theta$ possesses some fairly useful closure properties, we now show that it fails to satisfy even weak fragments of Replacement.  In what follows, $\Sigma_n(X)$ formulas are $\Sigma_n$ formulas with parameters from $X$ (so $\Sigma_1(\{\kappa\})$ formulas have at most parameter $\kappa$).  Recall a transitive set $M$ is \emph{admissible} if $(M,{\in})$ is a model of Kripke-Platek set theory, $\KP$ (see \cite{Barwise}); levels $J_\alpha$ of $L$ are admissible if and only if they satisfy $\Sigma_1$-Replacement (see \cite{Jensen}, Lemma 2.11).  $\Delta_1$-Comprehension is a consequence of $\KP$ (\cite{Barwise}, Theorem 4.5).

\begin{Theorem}\label{thm:noComp} $L_\theta$ is not admissible.  Indeed, we have
  \begin{enumerate}
    \item There is a cofinal map $F: \omega \to \theta$ that is $\Sigma_1(\{\kappa\})$-definable over $L_\theta$;
    \item The $\Sigma_1(\{\kappa\})$-theory of $L_\theta$ is not an element of $L_\theta$;
    \item There is an $a \subseteq \omega$ that is $\Delta_1(\{\kappa\})$-definable over $L_\theta$, but not in $L_\theta$.
  \end{enumerate}
\end{Theorem}

\begin{proof}
Let $F(0) = \kappa$, and inductively let $F(n+1)$ be the least $\alpha$ such that every wellfounded tree $T$ on $\kappa$ belonging to $J_{F(n)+1}$ has a rank function belonging to $J_\alpha$.  Note that this $F$ is well-defined: Given $\beta$ with $\kappa < \beta < \theta$, the set $W = \{T \in J_\beta \mid T$ is wellfounded$\}$ is an element of $J_{\theta}$ by $\Sigma_0$-Comprehension, and since $W$ has size $\kappa$ we may regard the direct sum of these, 
  \[
  \bigoplus_{T \in W} T = \{ \la T \ra^\frown s \mid s \in T \in W\},
  \]
 as a tree on $\kappa$; this is wellfounded, and by the definition of $L_\theta$, must be ranked in $L_\theta$, hence in some $J_\alpha$ with $\alpha<\theta$.

The map $F$ is easily seen to be increasing, and is defined via a $\Sigma_1(\{\kappa\})$-recursion over $L_\theta$, so is itself $\Delta_1(\{\kappa\})$-definable.

Now let $\lambda = \sup F" \omega$.  Then $J_\lambda$ satisfies that every wellfounded tree on $\kappa$ is ranked, and so we must have $\lambda = \theta$, by the minimality condition in the definition of $\theta$.  This proves (1); since there is a $\Sigma_1$-definable cofinal map $F: \om \to \theta$, we have that $L_\theta$ is not admissible.

Now if $H$ is the $\Sigma_1$ Skolem hull in $L_\theta$ of $\{\kappa\}$, we have that $H$ must contain the range of $F$, and in particular $\ON \cap H$ is unbounded in $\theta$; and $\kappa$ clearly is a regular cardinal in $H$, is the largest cardinal there, and $H \models \psi(\kappa)$.  Since $H \cong J_\alpha$ for some $\alpha \leq \theta$ by Condensation, we must have $\alpha = \theta$ by the minimality in our definition of $\theta$.  It follows that $H = L_\theta$; that is, the $\Sigma_1$-projectum of $L_\theta$ in parameter $\{\kappa\}$ is $\omega$.  In particular, the $\Sigma_1(\{\kappa\})$ theory of $L_\theta$ does not belong to $L_\theta$ (e.g.\ by acceptability, if this theory did belong to $L_\theta$, then it would belong to $L_\kappa$, and taking a transitive collapse there of the atomic diagram coded by the theory would yield the contradiction $L_\theta \in L_\kappa$).  This shows (2).

(3) now follows from (1) and (2): Fix an enumeration $\la \varphi_i(v) \ra_{i \in \om}$ of $\Sigma_1$ formulae with one free variable.  Since $F$ is $\Delta_1(\{\kappa\})$-definable, so is the set
 \[
   a = \{ 2^i 3^n \mid J_{F(n)} \models \varphi_i(\kappa) \}.
 \] 
The $\Sigma_1(\{\kappa\})$-theory of $L_\theta$ is $\Sigma^0_1(a)$; hence $a \notin L_\theta$, proving (3).
\end{proof}

\section{Determinacy separations in the $L_\theta$'s}\label{sec:ClopSep}

In this section, we fix a suitable $\Pi_1$ formula $\psi(u)$ and suppose $\theta = \theta^{\psi}$ and $\kappa$ are as in Definition~\ref{def:theta}.
\begin{Theorem}\label{thm:ClopdetinTheta} $L_\theta$ is a model of $\clopdetk$.
\end{Theorem}
The proof is the usual transfinite inductive construction; the point is the induction only needs to go up to the rank of the tree $T$, and $L_\theta$ is sufficiently closed to carry this out.
\begin{proof}[Proof of Theorem~\ref{thm:ClopdetinTheta}]
  Fix a tree $T \in L_\theta$ on $\kappa$ that is clopen.  There is a rank function $\rho$ for $T$ in $L_\theta$, so suppose $\rho(\la \ra) = \mu < \theta$.  By induction on the rank of nodes in $T$, we define a map $\sigma : T \to \{0,1\}$, required to satisfy
   \[
     \sigma(s) = 0 \iff (\exists \xi \in \kappa) s^{\frown} \la \xi \ra \in T \text{ and }\sigma(s^{\frown} \la \xi \ra) = 1.
   \]
So in particular, $\sigma(s) = 0$ for all terminal nodes of $T$, and $\sigma(s)$ is uniquely defined whenever it has been defined on all extensions of $s$ in $T$; so the map $\sigma$ is uniquely determined.  This $\sigma$ easily defines a winning strategy for the player (``Us'') whose turn it is at $s$, whenever $\sigma(s)=0$: Let Us play the least $\xi$ so that $\sigma(s^{\frown}\la \xi \ra) = 1$.  Then the opponent (``Them'') will only be able to play to nodes assigned $0$, and Us's favorable position is preserved; this continues until Us makes a move to a terminal node, at which point Them loses.

Since a winning strategy in $T$ is $\Sigma_0$-definable from $\{T, \sigma,\kappa\}$, it will be enough to show $\sigma \in L_\theta$.  Suppose $T \in J_\beta$, where $\beta < \theta$.  

Let us say a partial map $\tau : T \plto \{0,1\}$ is a \emph{partial winning strategy} if 
  \begin{itemize}
    \item $\tau(s) = 0 \to (\exists \xi < \kappa) \tau(s^\frown \la \xi \ra) = 1$;
    \item $\tau(s) = 1 \to (\forall \xi < \kappa) s^\frown\la \xi \ra \in T \to \tau(s^\frown \la \xi \ra) = 0$. 
  \end{itemize}
In particular, if $\tau(s)=1$ then $\tau$ is defined at all one-step extensions of $s$ in $T$.

Notice that by wellfoundedness of $T$, any two partial winning strategies must agree on the intersection of their domains.  We argue by induction on $\alpha \leq \mu$ that there is a partial winning strategy $\sigma_\alpha \in J_{\beta + \alpha + 1}$ so that $\{s \in T \mid \rho(s) < \alpha\} \subseteq \dom(\sigma_\alpha)$; in particular, $\sigma = \sigma_\mu \in J_{\beta+ \mu + 1} \subset L_\theta$ (as $\beta + \mu + 1 < \theta$ by Lemma~\ref{lem:thetaordclosure}), as needed.

The claim is clear for $\sigma_0 = \emptyset$.  Given $\sigma_\alpha$, we may let $\sigma_{\alpha+1}(s)$ be defined iff for all $\xi$ with $s^{\frown}\la \xi \ra \in T$, $\sigma_\alpha(s^\frown \la \xi \ra)$ is defined; and then set $\sigma_{\alpha+1}(s) = 0$ iff for some such $\xi$, $\sigma_{\alpha}(s^\frown\la \xi\ra) = 1$.  This is clearly definable from $\sigma_\alpha$, and inductively, takes values on all nodes $s$ with $\rho(s) \leq \alpha$, so takes care of the successor case.  

So let $\lambda \leq \mu$ be limit and suppose we have the claim for $\alpha< \lambda$.  Put
  \[
    \sigma_\lambda = \bigcup \{ \tau \in J_{\beta + \lambda} \mid \tau\text{ is a partial winning strategy}\}.
  \]
It is easy to see that $\sigma_\lambda$, so defined, is a partial winning strategy, and by inductive hypothesis, takes values at all nodes $s$ with $\rho(s) < \lambda$.  It is definable over $J_{\beta + \lambda}$, so belongs to $J_{\beta + \lambda + 1}$.  This completes the proof.
\end{proof}

The next theorem is the heart of the paper.

\begin{Theorem}\label{thm:NoOpen} $L_\theta$ is not a model of $\opdetk$.
\end{Theorem}
\begin{proof} Note that if $L_\theta$ satisfies the determinacy of the open game played on a tree $T$, then the witnessing strategy being winning is upwards absolute to $V$ (recall $T \subseteq \kappa^{<\om}$ is a countable tree in $V$): A winning strategy for Open is a tree with no infinite branches, and since $L_\theta$ ranks its wellfounded trees, such a strategy must really have no infinite branch in $V$; on the other hand, a winning strategy for Closed is just a strategy with no terminal nodes, and this is clearly absolute.

So it will be sufficient to exhibit a game tree $T$ on $\kappa$---equivalently, on $L_\kappa$---which is winning for Closed in $V$, but for which no true winning strategy can belong to $L_\theta$.  The game we define is one in which any winning strategy for Closed must compute the $\Sigma_1(\{\kappa\})$ theory of $L_\theta$; by (2) of Theorem~\ref{thm:noComp}, there can be no winning strategy for Closed in $L_\theta$.

Let us describe in outline the game and the proof that it has the desired properties.  In the first round, Open chooses a natural number, $i$, corresponding to the index of some $\Sigma_1(\{\kappa\})$ formula, $\varphi_i(\kappa)$, in our fixed enumeration.  Closed must choose a truth value $j \in \{0, 1\}$ for this $\varphi_i(\kappa)$.  In the remaining $\omega$ moves of the game, Closed must plays ordinals $\alpha_0, \alpha_1, \dots$ corresponding to levels $J_{\alpha_n}$ that in some sense resemble $L_\theta$ while agreeing with the value $j$ played for $\varphi_i(\kappa)$; meanwhile, Open is allowed to play arbitrary sets $x_0,x_1,\dots$ from our universe of sets $L_\kappa$, which Closed is obligated to capture by the models $L_{\alpha_k}$.  Closed must also ensure these levels cohere, by providing $\Sigma_0$-preserving embeddings $\pi_{i,{i+1}} : L_{\alpha_i} \to L_{\alpha_{i+1}}$.

Now, Closed can win the game in $V$, simply by always playing the correct truth value of $\varphi_i(\kappa)$ in $L_\theta$, then by playing $\alpha_i$ corresponding to transitive collapses of sufficiently large hulls of initial segments of $L_\theta$.  But any true winning strategy in this game must play the correct truth values of $\varphi_i(\kappa)$: Using the countability of $L_\kappa$ in $V$, we can, playing as Open, list \emph{all} sets in $L_\kappa$.  If Closed lasts infinitely many plays, then we may take the direct limit along the embeddings $\la \pi_n \ra_{n\in\om}$; we argue that this direct limit must be isomorphic to $L_\theta$, so that the truth value $j$ asserted by Closed must have agreed with the real truth value of $\varphi_i(\kappa)$ in $L_\theta$.

We proceed to the formal argument.  Let the game tree $T$ consist of finite plays of the following form:
\[\begin{array}{l|ccccccc}
 \text{Open}     & i \in \om &                    & x_0 &            & x_1 & & \dots\\
 \hline 
 \text{Closed}      &   & j \in \{0,1\}, \alpha_0 &     & \pi_0, \alpha_1 &     & \pi_1, \alpha_2 & \dots
\end{array} 
 \; \; x_n, \alpha_n, \pi_n \in L_\kappa 
\]
The Closed player must maintain the following conditions, for all $n\in \om$: 
  \begin{itemize}
    \item Each $J_{\alpha_n}$ satisfies: ``there is a largest cardinal $\kappa_n$, which is regular, and $\psi(\kappa_n)$ holds'';
    \item $J_{\alpha_n} \models \varphi_i(\kappa_n)$ iff $j = 0$;
    \item $x_n \in J_{\kappa_{n+1}}$, for all $n$;
    \item $\pi_{n}: J_{\alpha_n} \to J_{\alpha_{n+1}}$ is $\Sigma_0$-preserving, $\crit(\pi_n) = \kappa_n$, and $\pi_{n}(\kappa_n) = \kappa_{n+1}$;
    \item For any tree $S \in J_{\alpha_n}$ on $\kappa_n$, $\pi_{n}(S)$ is either ranked or illfounded in $J_{\alpha_{n+1}}$
  \end{itemize}
(cf.\ Definition~\ref{def:theta}).  The game tree $T$ is easily seen to be definable over $L_{\kappa}$, so belongs to $L_\theta$.

\begin{Lemma}\label{lem:Closedwins} Closed has (in $V$) a winning strategy for the game on $T$.
\end{Lemma}
\begin{proof}[Proof of Lemma~\ref{lem:Closedwins}]
We describe how Closed ought to play to win the game in $T$.  For notational convenience, fix a sequence $\la i, x_0, x_1, \dots\ra$ potentially played by Open; the replies by Closed will always depend only on the moves made so far.  Let $j=0$ iff $L_\theta \models \varphi_i(\kappa)$, and fix $\tau_0 < \theta$ sufficiently large that $J_{\tau_0} \models \varphi_i(\kappa)$ iff $j=0$; since $\varphi_i$ is $\Sigma_1$, such an ordinal exists.

Having fixed $\tau_0$, let $\la \tau_n \ra_{n\in\om}$ be an increasing sequence of ordinals cofinal in $\theta$ so that for all $n$, every tree on $\kappa$ in $L_{\tau_n}$ that is wellfounded is ranked in $L_{\tau_{n+1}}$; note that such a sequence exists, by the same argument that the $F$ in (1) of Theorem~\ref{thm:noComp} is well-defined.

We now define an increasing $\om$-sequence of sets $H_0 \subseteq H_1 \subseteq H_2 \subseteq \dots \subset L_\theta$, with each $H_n \in L_\theta$, by induction.  Let $H_{-1} =\emptyset$, and for each $n<\om$, let $H_{n}$ satisfy
 \begin{itemize}
   \item $|H_n| < \kappa$ in $L_\theta$,
   \item $H_{n} \prec J_{\tau_n}$,
   \item $H_{n-1} \cup \{x_0, \dots, x_{n-1}\} \subseteq H_{n}$,
   \item $H_{n} \cap \kappa \in \kappa$.
 \end{itemize}
The fact that $\kappa$ is regular and uncountable inside $L_\theta$ allows us to obtain each $H_n$ in the standard way, by interleaving Skolem hulls and transitive closures below $\kappa$ for $\omega$-many steps; and then each $H_n$, being inductively defined from finitely elements of $L_\theta$, is in $L_\theta$.  Note that the third point entails $H_n \cap L_\kappa = L_{\kappa_n}$ for some $\kappa_n < \kappa$.

Now for $n < \om$, let $\alpha_n$ be the unique ordinal so that $J_{\alpha_n} \cong H_n$, which exists by Condensation.  Furthermore, the fact that $|H_n|<\kappa$ in $L_\theta$ guarantees $\alpha_n < \kappa$; and letting $e_n: J_{\alpha_n} \to H_n \prec J_{\tau_n}$ be the anticollapse embedding, we have $\crit(e_n) = \kappa_n$.  We then set, for $n < \om$,
  \[
    \pi_n = e_{n+1}^{-1} \circ e_n.
  \]
It is now easily verified that $\la \alpha_n, \pi_n \ra_{n\in\om}$ satisfies the requirements of the game on $T$ for Closed; $\Sigma_0$-preservation of the maps $\pi_n$ is immediate from elementarity of the embeddings and the fact that for all $n$, $J_{\tau_n}$ and $J_{\tau_{n+1}}$, hence $H_n$ and $H_{n+1}$, have the same $\Sigma_0(H_n)$-theory.  
\end{proof}
\begin{Lemma}\label{lem:truththeta} Suppose $\sigma$ is a winning strategy for Closed in $T$.  Then for all $i$, $\sigma(\la i \ra) = 0$ iff $L_\theta \models \varphi_i(\kappa)$.
\end{Lemma} 
\begin{proof}[Proof of Lemma~\ref{lem:truththeta}]
  Let $\sigma$ be winning for Closed, and let $i \in \om$ be arbitrary.  Working in $V$, let $\la x_n \ra_{n\in\omega}$ be an enumeration of $L_\kappa$.  Then $\sigma$ produces the sequence $\la \alpha_n, \pi_n \ra_{n \in \omega}$ in response to play of $\la i \ra^{\frown}\la x_n \ra_{n\in\om}$ by Open.  We may compose the maps $\pi_n$ to obtain a commuting system of maps $\pi_{m,n}: L_{\alpha_m} \to L_{\alpha_n}$ for $m<n$, resulting in a directed system
   \[
     \la L_{\alpha_n}, \pi_{n,m} \mid n < m < \omega \ra.
   \]
Let $(N, \epsilon)$ be the direct limit obtained, and for all $n$ let $\pi_{n,\infty} : L_{\alpha_n} \to N$ be the direct limit embedding.  Note that each $\pi_{n,\infty}$ is $\Sigma_0$-preserving.  As per usual, we identify the wellfounded part of the model $N$ with its transitive collapse.  Let $\kappa_\infty = \pi_0(\kappa_0)$.
\begin{Claim}\label{claim:critinf} $\kappa_\infty = \kappa$.
\end{Claim}
\begin{proof}[Proof of Claim~\ref{claim:critinf}]
  First observe that the rule stipulating $\pi_n(\kappa_n) = \kappa_{n+1}$ guarantees $\kappa_\infty = \pi_{n,\infty}(\kappa_n)$ for all $n$, and therefore $\crit(\pi_{n,\infty}) = \kappa_n$ for all $n$.  Now, suppose $\xi < \kappa$.  Then there is some $n$ so that $x_n = \xi$, and so $\xi < \kappa_{n+1}$; since $\crit(\pi_{n+1}) = \crit(\pi_{n+1,\infty}) = \kappa_{n+1}$, we have $\pi_{n+1,\infty}(\xi) = \xi < \kappa_\infty$.  So $\kappa \subseteq \kappa_\infty$.
  
  Now suppose $\xi \mathrel{\epsilon} \kappa_\infty$ in $N$.  Then by definition of direct limit, we have some $n$ and $\bar{\xi} \in L_{\alpha_n}$ so that $\pi_{n,\infty}(\bar{\xi}) = \xi$.  But then clearly $\bar{\xi} < \kappa_n = \crit(\pi_{n,\infty})$, so that $\bar{\xi}=\xi$, and hence $\kappa_\infty \leq \kappa$ as needed.
\end{proof}
  Let us denote the wellfounded ordinal height of $N$, $\ON \cap N$, by $\wfo(N)$.  Note that this makes sense whether or not $N$ is illfounded.  An elementary argument using the $\Sigma_0$ definition of ordinal multiplication shows $\om \cdot \wfo(N) = \wfo(N)$.
\begin{Claim}\label{claim:Lsallintheta} $J_{\wfo(N)} \subseteq N$.
\end{Claim}
\begin{proof}[Proof of Claim~\ref{claim:Lsallintheta}]
Let $(\exists u)\Phi(u,v,w)$, where $\Phi$ is a $\Sigma_0$ formula, be the uniform $\Sigma_1$ definition of the graph of the function $\alpha \mapsto J_\alpha$.  It is immediate from the $\Sigma_0$-preservation of the maps $\pi_{n,\infty}$ and the definition of direct limit that for all $N$-ordinals $\alpha$, there is a unique $N_\alpha \mathrel{\epsilon} N$ so that $N \models (\exists u)\Phi(u,\alpha,N_\alpha)$.

We argue by induction that for $\alpha < \wfo(N)$ we have $N_\alpha = J_\alpha$.  This is clear for successors, using the fact that $J_{\alpha+1}$ is the rudimentary closure of $J_\alpha \cup \{J_\alpha \}$, since rudimentary functions are $\Sigma_0$-definable.

For limit $\lambda$, let $n$ be sufficiently large that there is $\bar{\lambda} \in J_{\alpha_n}$ with $\pi_{n,\infty}(\bar{\lambda}) = \lambda$.  Then by $\Sigma_0$-preservation of $\pi_n$, we must have $\pi_{n,\infty}(J_{\bar{\lambda}}) = N_\lambda$, and so
  \[
    N \models N_\lambda = \bigcup \{ z \in N_\lambda \mid (\exists x \in N_\lambda)\Phi(x, \alpha, z) \}
  \]
Thus by inductive hypothesis, we must have $N_\lambda = \bigcup_{\alpha<\lambda} J_\alpha = J_\lambda$.
\end{proof}
\begin{Claim}\label{claim:wfobig} $\theta \leq \wfo(N)$.
\end{Claim}
\begin{proof}[Proof of Claim~\ref{claim:wfobig}] We have already shown $\kappa \in \wfo(N)$.  Suppose towards a contradiction that $\wfo(N) < \theta$.  By minimality of $\theta$, there must be a tree $S \in J_{\wfo(N)}$ on $\kappa$ that is wellfounded, but has rank greater than $J_{\wfo(N)}$.  By the previous claim, this $S$ belongs to $N$.  Let $n$ be sufficiently large that $S = \pi_{n,\infty}(\bar{S})$ for some tree on $\kappa_n$ in $J_{\alpha_n}$.  Then by the rules of the game tree $T$, there is a ranking function $\bar{\rho}$ for $\pi_{n}(\bar{S}) \in J_{\alpha_{n+1}}$.  By $\Sigma_0$-preservation of the direct limit map, we have $S = \pi_{n,\infty}(\bar{S})$ is ranked by the map $\rho = \pi_{n+1,\infty}(\bar{\rho})$ in $N$.  Now let $x \in S$ be a node with $\rho(s) = \wfo(N)$.  Then $\wfo(N) = \rho" T_s$, which must belong to $N$ by $\Sigma_0$ elementarity; but this contradicts the definition of $\wfo(N)$ as the least ordinal not in $N$.
\end{proof}
\begin{Claim}\label{claim:wfosmall} $\wfo(N) \leq \theta$.
\end{Claim}
\begin{proof}[Proof of Claim~\ref{claim:wfosmall}]
Otherwise, we would have $\theta = \pi_{n,\infty}(\bar{\theta})$ for some $n$ and $\bar{\theta} < \alpha_n$.  But then $L_\theta \subseteq N$ combined with $\Sigma_0$-elementarity implies $L_{\bar{\theta}}$ satisfies the defining properties of $L_\theta$, i.e.\ $L_{\bar{\theta}} \models ``\kappa_n$ is the largest cardinal, is uncountable, $\psi(\kappa_n)$ holds, and all wellfounded trees on $\kappa_n$ are ranked''.
But $\bar{\theta} < \theta$, contradicting minimality.
\end{proof}
The only remaining possibility is $\wfo(N) = \theta$.
\begin{Claim}\label{claim:Nwelf} $N$ is wellfounded.
\end{Claim}
\begin{proof}[Proof of Claim~\ref{claim:Nwelf}] Suppose towards a contradiction that $N$ is illfounded.  By the previous two claims, $\wfo(N) = \theta$.  Fix a nonstandard ordinal $b \mathrel{\epsilon} N$.  Working in $(L_b)^N$, consider the map $F: \omega \to \theta$ defined in the proof of Theorem~\ref{thm:noComp}.  Recall $F$ was cofinal in $\theta$ and $\Delta_1(\{\kappa\})$-definable over $L_\theta$; moreover, since $(L_b)^N$ is an end-extension of $L_\theta$, this map is computed correctly in $(L_b)^N$, and so $F \mathrel{\epsilon} (L_{b+1})^N$.  But then $\theta = \sup F"\omega \mathrel{\epsilon} N$, a contradiction.
\end{proof}

So $N$ is wellfounded.  But it is easy to argue that $N$ is a model of $(\forall x)(\exists \beta) x \in J_\beta$, from which we conclude that $N = L_\theta$ by the proof of Claim~\ref{claim:Lsallintheta}.  Now notice that since $J_\theta$ is obtained as the direct limit of the $J_{\alpha_n}$ under $\Sigma_0$-preserving embeddings, and since all of the $J_{\alpha_n}$ agree as to the truth value of the $\Sigma_1$ statements $\varphi_i(\kappa_n)$, we have that $j=0$ iff $N = J_\theta \models \varphi_i(\kappa)$.
\end{proof}
This completes the proof of Lemma~\ref{lem:truththeta}, and by the remarks preceding, we have proved Theorem~\ref{thm:NoOpen}.
\end{proof}

\section{Separation results in typed theories}\label{sec:Frags}

The previous section proved separation results in the context of weak fragments of set theory with a largest cardinal $\kappa$.  We now carry these over to the setting of set theories with proper classes.

We regard these theories as formulated in a language with two typed variables.  For background on von Neumann-Bernays-G\"{o}del set theory ($\NBG$), the reader may consult Mendelson's text \cite{Mendelson}; Jech's book \cite{Jech} has a listing of the axioms.  Our models of $\NBG$ will have the form $(V, {\in}, \C)$, so that $V$ is the universe of sets and $\C$ is the collection of subclasses of $V$.  The theory of games on proper class-sized trees is developed in $\NBG$ in the natural way; see \cite{GitHam} for details.

\begin{Theorem}\label{thm:theNBGmodel} Assume there is a wellfounded model of $\ZF^{-} + ``(\exists \kappa)\kappa$ is inaccessible''.  Then there is a model of $\NBG$ in which clopen determinacy for proper class games holds, while open determinacy for class games fails.
\end{Theorem}
Obviously there are weaker assumptions which suffice, but this is at least a natural one.
\begin{proof}
Let $\psi$ be the formula
 \[
  (\forall \alpha < \kappa)(\exists \mu < \kappa) \alpha < \mu \wedge L_\kappa \models \mu \text{ is a cardinal}.
 \]
This is a $\Pi_1$ formula in parameter $\kappa$, and clearly if $M \models \ZF^{-} + ``\lambda$ is inaccessible'', we have $M \models \psi(\lambda)$.  Condensation arguments imply that $\theta^{\psi}$ (Definition~\ref{def:theta}) exists; as before, let $\kappa$ be the largest cardinal of $L_{\theta} = L_{\theta^{\psi}}$.  Since $L_{\theta^{\psi}}$ contains no bounded subsets of $\kappa$ not in $L_\kappa$, we have that $\kappa$ is inaccessible in $L_\theta$.

We have an obvious way of regarding $L_\theta$ as a model of $\NBG$.  Namely, let $\C_\theta = \P(L_\kappa) \cap L_\theta$, and put $M_\theta = (L_\kappa, {\in}, \C_\theta)$.  The proper classes of $M_\theta$ are precisely the elements of $\C_\theta \setminus L_\kappa$.  Closure of $L_\theta$ under $\Sigma_0$-Comprehension in parameters ensures that $M_\theta$ satisfies the Class Comprehension schema (since quantification over sets in $M_\theta$ is equivalent to bounded quantification by $L_\kappa$).  Note also that any two proper classes of $M_\theta$ are in bijection, and $M_\theta$ satisfies the existence of a universal choice function. 

Combining Theorems~\ref{thm:ClopdetinTheta} and \ref{thm:NoOpen}, we are done.
\end{proof}

\begin{Remark}\label{rem:Adm}
  Working in appropriate strengthenings of $\NBG$, one can develop a theory of $L$ for proper class wellorders of order-type larger than $\ON$, thus obtaining what Gitman and Hamkins call a ``meta-$L$ structure'' $L_{\Theta}$, where $\Theta$ is the supremum of the order-types of proper class well-orders of $\ON$.  These authors ask \cite{GitHam} whether clopen determinacy for class games in $\NBG$ is enough to prove admissibility of this meta-$L$ structure.
 
  We see the answer is no: While the $M_\theta$ of Theorem~\ref{thm:theNBGmodel} satisfies clopen determinacy for class games, its meta-$L$ is precisely $L_\theta$, which is inadmissible.  However, we prove in forthcoming work that clopen determinacy for class games implies the existence of admissible models of ``there is an inaccessible cardinal'', so that consistency strength-wise, the determinacy assumption is the stronger theory. 
  \end{Remark}

Taking $\psi(u)$ to be the formula ``all sets in $L_u$ are countable in $L_u$'', we similarly obtain the main result of Schweber's \cite{Schweber} (see that paper for background on the weak base theory $\RCA$ for third order arithmetic):

\begin{Theorem} $(\omega, \R \cap L_\theta, \omega^{\R \cap L_\theta} \cap L_\theta)$ is a model of $\RCA$ satisfying $\clopdetR$, but not $\opdetR$.
\end{Theorem}

Various other results may be formulated for languages with at least two ``top types'' beyond $0$, e.g.\ by letting $\psi$ have the intended meaning ``$\kappa = \omega_5$'' (7 types) or ``$\kappa = \omega_{\omega+1}$'' ($\omega+2$ types), to give those examples mentioned in the abstract.  We just point out that the structures $L_{\theta^{\psi}}$ always satisfy $\Sigma_0$-Comprehension, and so satisfy full Comprehension for formulas whose quantifiers range over objects of all but the largest type; thus these structures ought to convert to models of any reasonable base theory in higher order arithmetic.

\section{On Borel class determinacy in Morse-Kelley set theory}
Throughout this section, we let the stratified Borel hierarchy $\la \Sigma^0_\alpha \ra_{\alpha<\omega_1}$ of subsets of $X^\om$, for various sets $X$, be defined in the natural way, by letting $\Sigma^0_1$ be the class of complements of sets of the form $[T]$ in $X^\om$ for trees $T$ on $X$, and iterating under countable union and complement.

We answer one further question of Gitman and Hamkins \cite{GitHam}, as to whether Borel determinacy for proper class games is provable in the stronger set theory Morse-Kelley, or $\MK$ (about which see e.g.\ \cite{Mendelson}).  As one might expect, the answer is no: We demonstrate the existence of a proper class game with a $\Sigma^0_4$ winning condition which is not provably determined in $\MK$, in close analogy with H.\ Friedman's important result \cite{Fr} in the context of second order arithmetic.

Thankfully we need not delve into a proof of Friedman's theorem here, and may take the following strengthening due to Martin \cite{MaBook} as a black box (we present a complete proof of the lightface case in \cite{Ha}, pp.\ 18-19).  Here $\ZF^{-}$ is $\ZF$ with the Axiom of Power Set dropped.

\begin{Theorem}[Friedman, Martin]\label{thm:Friedman}  Let $z$ be a real, and suppose $\beta$ is minimal so that $L_\beta[z]$ is a model of $\ZF^{-} + V{=}L[z]$; then there is a $\Sigma^0_4(z)$ game for which no winning strategy belongs to $L_\beta[z]$.
\end{Theorem}

\begin{Theorem}\label{thm:classFriedman} Suppose there is a transitive set model of $\ZF^{-}+(\exists \kappa)``\kappa$ is inaccessible''.  
%
Then there is a sequence of class trees, $\la T_{i,j,k} \ra_{i,j,k \in \om}$, definable by a quantifier-free formula with no parameters, so that 
$\bigcup_{i\in\om}\bigcap_{j\in\om}\bigcup_{k\in\om} [T_{i,j,k}]$ is the payoff class of a game whose determinacy cannot be proved in $\MK$.
\end{Theorem}

\begin{proof} Let $\kappa$ be the least so that some $\beta >\kappa$ exists such that $L_{\beta} \models \ZF^{-} + ``\kappa$ is inaccessible.''  As before, the theorem is proved by reinterpreting this set model as a two-sorted model $M = (L_\kappa, {\in}, \P(L_\kappa) \cap L_\beta)$. Note that this time, we have that the full Class Comprehension schema (allowing formulae with class quantifiers) holds in $M$, because $L_\beta$ is a model of the stronger theory $\ZF^{-}$.

The result is more or less immediate from the following:
\begin{Theorem}[Martin]\label{exercise}
  Let $\beta > \kappa$ be the least ordinal above $\kappa$ so that $L_\beta \models \ZF^{-}$.  Then there is a $\Sigma^0_4$ game on $\kappa^{<\om}$ which has no winning strategy in $L_\beta$.
\end{Theorem}
This is Exercise 1.4.5 in Martin's unpublished determinacy manuscript \cite{MaBook}.  For the reader's convenience, we offer our solution.

\begin{proof}[Proof of Theorem~\ref{exercise}]  Regard conditions in the collapse poset $\Coll(\omega,\alpha)$ as finite sequences $p : n \to \alpha$.  For $G$ $\Coll(\omega,\alpha)$-generic, we let $z_G$ denote the collapse real coded by $G$; more generally, any $h:\om \to \ON$ induces a prewellorder $z_h$ of $|p|\in\om$.
  
  For the rest of the proof, fix $G$ $\Coll(\om,\kappa)$-generic over $L_\beta$.  Then $L_\beta[G]$ is the minimal transitive model of $\ZF^{-}$ containing the real $z_G$.  By Theorem~\ref{thm:Friedman}, $\Sigma^0_4(z_G)$ determinacy fails in $L_\beta[G]$.  So let
  \[
  A_G := \{x \in \om^\om  \mid  \exists i\forall j\exists k\forall n \varphi(i,j,k,x \rst n, z_G \rst n)\},
  \]
where $\varphi$ is a recursive condition, be a $\Sigma^0_4(z_G)$ set witnessing this failure.

Working in $L_\beta$, define a game on $\kappa^{<\om} \times \om$ (easily converted to one on $\kappa$ via tuple coding) as follows: Moves in the game are pairs $\la p, x \ra$ where $p \in \Coll(\om,\kappa)$ and $x \in \om$.  The players must maintain $p_{n+1} \lneq p_n$ for all $i$.  Player I wins if the infinite play $\la h, x\ra := \la \bigcup_n p_n, \la x_n \ra_{n\in\om} \ra$ satisfies
  \[
    \exists i \forall j \exists k \forall n \varphi(i,j,k, x \rst n, z_h \rst n).
  \]
This clearly defines a $\Sigma^0_4$ subset of $\kappa^\om$.

We claim there is no winning strategy in this game in $L_\beta$.  Suppose without loss of generality that $\sigma$ is a winning strategy for Player I.  We will obtain a contradiction by converting this winning strategy to one in $L_\beta[G] = L_\beta[z_G]$ that wins the game on $\om$ with payoff $A_G$.

\begin{Claim}\label{claim:GenExt}
  Suppose $s = \la p_0, x_0, \dots p_n, x_n \ra$ is a position reachable by $\sigma$ where Player I has moved last, and $k \in \om$; then the set
    \[
      D_{\la s, x\ra} = \{ q \mid (\exists p)\sigma( \la p_0, x_0, \dots, p_n, x_n, p, x_{n+1}\ra) = \la q, l \ra \text{ for some }l \in \om \}
    \] 
is dense in $\Coll(\om,\kappa)$ below $p_n$.
\end{Claim}
\begin{proof} This is immediate: any pair $\la p, k \ra$ with $p \lneq p_n$ constitutes a legal move for Player II, and $\sigma$ must reply with $q \leq p$.
\end{proof}

Now in $L_\beta[G]$, it is now easy to describe a strategy for Player I to win the game with payoff $A_G$: copy the moves in the game on $\om$ to one in the game on $\kappa^{<\om} \times \om$, and use Claim~\ref{claim:GenExt} to attribute moves $p_{2n+1}$ to Player II. 

In more detail, assume inductively that we have reached a position $\la x_0,\dots, x_{2n} \ra$ in the game on $\om$, so that some position $s =\la x_i, p_i \ra_{i < 2n+1}$ in the larger game is according to $\sigma$, and $p_{2n} \in G$.  We respond to Player II's next move $x_{2n+1}$ using the Claim to find $p_{2n+1} \in D_{\la s, x_{2n+1} \ra} \cap G$; then $\sigma$ makes a reply $p_{2n+2},x_{2n+2}$ with $p_{2n+2} \in G$, and the construction continues.  

Thus any play $x \in \om^\om$ according to the strategy we have described may be copied to a play $\la G, x\ra$ compatible with $\sigma$.  Since $\sigma$ is a winning strategy for Player I, we must have $x \in A_G$.  But the strategy we have described is clearly definable in $L_\beta[G]$, which completes the contradiction.
\end{proof}

We finish by defining $S$ to be the class of sequences of pairs recording, in an increasing fashion, a potential result of a partial play in the game from the previous proof (recall $\kappa = \ON^M$), and setting
 \[
  T_{i,j,k} = \{ s \in S \mid (\forall n < |s|) \varphi(i,j,k, s_n(0),s_n(1)) \}.
 \]
This completes the proof of Theorem~\ref{thm:classFriedman}.
\end{proof}

Thus Borel class games need not be determined in $\MK$.  We can say a bit more:  Friedman's level-by-level analysis of the strength of Borel determinacy beyond $\Sigma^0_4$ likewise may be applied to the setting of proper class games, by starting with a model with an inaccessible $\kappa$ in which $V_{\kappa+\alpha}$ exists and collapsing $\kappa$ to be countable.  We have:

\begin{Theorem}
  Work in $\ZFC$, and assume there is an inaccessible cardinal.  Let $\alpha < \omega_1$.  Then there is a transitive model $M$ of $ZF^{-} + $``There is an inaccessible cardinal $\kappa$ such that $V_{\kappa+\alpha}$ exists, and $\Sigma^0_{1+\alpha+3}$ determinacy for games on $\kappa$ fails.''
\end{Theorem}
  Regarding $(V_\kappa^M, {\in},  \P(V_\kappa^{M})\cap M)$ as an $\MK$-model, we obtain the expected hierarchy of strength: 
  
\begin{Theorem}  Over $\MK$, $\Sigma^0_\alpha$ determinacy for class games does not imply $\Sigma^0_\beta$ determinacy for class games, whenever $\alpha < \beta$.
\end{Theorem}

Sharper results are possible.  In the case of $\Sigma^0_{1+\alpha+3}$ games on $\omega$, the author has obtained equivalences between determinacy principles and the existence of countable wellfounded models of fragments of set theory; this was the setting in which the relevance of the models $L_\theta$ was first made apparent.  For example,

\begin{Theorem}[\cite{Ha}]
  In $\PCA$, $\Sigma^0_4$ determinacy (for games on $\om$) is equivalent to the existence of a $\theta$ so that $L_\theta \models ``\omega_1$ exists, and all wellfounded trees on $\omega_1$ are ranked.''
\end{Theorem}

The results of this section suggest that these proofs may be uniformly translatable from the setting of second order arithmetic to that of Morse-Kelley set theory; once the theory of ``meta-$L$'' has been developed for proper class wellorders of ordertype $\Gamma > \ON$ so that $L_\Gamma$ makes sense, it is only necessary to check that the arguments in the style of Friedman and Martin in \cite{Ha} go through.  We believe the following (appropriately formalized) is provable in $\MK$:

\begin{Conjecture}
The determinacy of all Borel proper class games is equivalent to the existence, for each class $A$ and countable ordinal $\alpha$, of a class coding a model $L_\Theta[A] \models \ZF^{-} + ``V_{\ON + \alpha}$ exists.''
\end{Conjecture}

\bibliographystyle{asl}
\bibliography{data1}
\end{document}